\documentclass{amsart}
\usepackage[foot]{amsaddr}
\usepackage{amsmath}
\usepackage{amsthm}
\usepackage{amssymb}
\usepackage{bm}
\usepackage[margin=0.5in]{caption}
\usepackage{dsfont}
\usepackage{enumitem}
\usepackage[margin=1.5in]{geometry}
\usepackage{graphicx}
\usepackage[hidelinks]{hyperref}
\usepackage[capitalise]{cleveref}
\usepackage{mathrsfs}
\usepackage{mathtools}
\usepackage{nameref}
\usepackage[new]{old-arrows}
\usepackage{spectralsequences}
\usepackage{thmtools}
\usepackage{xcolor}
\usepackage{xparse}

\usepackage{tikz}
\usetikzlibrary{cd}
\usetikzlibrary{positioning, angles, quotes}
\newcommand{\nospacepunct}[1]{\makebox[0pt][l]{\,#1}} %use this command to punctuate cd's

\theoremstyle{plain}
\newtheorem{thm}{Theorem}[section]
\newtheorem*{thm*}{Theorem}

\newtheorem{cor}[thm]{Corollary}
\newtheorem*{cor*}{Corollary}

\newtheorem{prop}[thm]{Proposition}
\newtheorem*{prop*}{Proposition}

\newtheorem{lem}[thm]{Lemma}
\newtheorem*{lem*}{Lemma}

\newtheorem*{ex*}{Example}

\newtheorem*{q*}{Question}

\newtheorem*{conj*}{Conjecture}

\theoremstyle{definition}
\newtheorem{defn}[thm]{Definition}
\newtheorem*{defn*}{Definition}

\theoremstyle{remark}
\newtheorem{rem}[thm]{Remark}
\newtheorem*{rem*}{Remark}

\theoremstyle{plain}
\newtheorem{manualtheoreminner}{Theorem}
\newenvironment{manualtheorem}[1]{%
  \renewcommand\themanualtheoreminner{#1}%
  \manualtheoreminner
}{\endmanualtheoreminner}

\newtheorem{manualpropinner}{Proposition}
\newenvironment{manualprop}[1]{%
  \renewcommand\themanualpropinner{#1}%
  \manualpropinner
}{\endmanualpropinner}

\newtheorem{manualcorinner}{Corollary}
\newenvironment{manualcor}[1]{%
  \renewcommand\themanualcorinner{#1}%
  \manualcorinner
}{\endmanualcorinner}

\Crefname{defn}{Definition}{Definitions}

\newcommand{\pres}[2]{\langle{#1}\mid{#2}\rangle}

\newcommand{\F}{\mathtt{F}}
\newcommand{\FP}{\mathtt{FP}}

\newcommand{\btwo}[1]{b^{(2)}_{#1}}

\newcommand{\ltwo}{\ell^2}

\newcommand{\C}{\mathbb{C}}
\newcommand{\N}{\mathbb{N}}
\newcommand{\Q}{\mathbb{Q}}
\newcommand{\R}{\mathbb{R}}
\newcommand{\Z}{\mathbb{Z}}

\newcommand{\inv}{^{-1}}

\newcommand{\ab}{^{\text{ab}}}

\DeclareMathOperator{\Isom}{Isom}

\DeclareMathOperator{\Ore}{Ore}

\DeclareMathOperator{\Tor}{Tor}

\newcounter{samcomments}

\title{Algebraic fibring of a hyperbolic $7$-manifold}
\author{Sam P. Fisher}
\email{sam.fisher@maths.ox.ac.uk}
\address{University of Oxford, United Kingdom}
\date{\today}

\begin{document}

\maketitle

\begin{abstract}
We show there is a finite-volume, hyperbolic $7$-manifold that algebraically fibres with finitely presented kernel of type $\mathtt{FP}(\Q)$. This manifold is a finite cover of the one constructed by Italiano--Martelli--Migliorini.
\end{abstract}

%%%                  %%%
%%%   Introduction   %%%
%%%                  %%%
\section{Introduction}\label{sec:intro}

A group $G$ \textit{algebraically fibres} if there is an epimorphism $G \longrightarrow \Z$ with finitely generated kernel, and a manifold \textit{algebraically fibres} if its fundamental group algebraically fibres. Let $F \longrightarrow M \longrightarrow S^1$ be a topological fibration of a manifold $M$ with connected fibre $F$. Then the low-dimensional terms of the long exact sequence of homotopy groups give
\[
    \begin{tikzcd}
\pi_2(S^1) \arrow[d, equal] \arrow[r] & \pi_1(F) \arrow[d, equal] \arrow[r] & \pi_1(M) \arrow[d, equal] \arrow[r] & \pi_1(S^1) \arrow[d, "\cong", no head] \arrow[r] & \pi_0(F) \arrow[d, equal] \\
1 \arrow[r]                                                  & \pi_1(F) \arrow[r]                                & \pi_1(M) \arrow[r]                                & \Z \arrow[r]                                                 & 1 \nospacepunct{,}                    
\end{tikzcd}
\]
so algebraic fibrations can be viewed as weaker versions of fibrations over the circle, provided $\pi_1(F)$ is finitely generated. In some cases, one can deduce the existence of a topological fibration from the existence of an algebraic fibration. For instance, Stallings showed that if $G$ is isomorphic to the fundamental group of a closed $3$-manifold $M$ and there is an algebraic fibration $G \longrightarrow \Z$, then it is induced by a fibration $M \longrightarrow S^1$ \cite{Stallings3mflds}. When $M$ is a smooth manifold of dimension $\geqslant 6$ and $\pi_1 M \cong \Z$, Browder and Levine showed that every algebraic fibration of $M$ is induced by a smooth fibration $M \longrightarrow S^1$, provided the homotopy groups of $M$ are all finitely generated \cite{BrowderLevineFibering1966}. For general $\pi_1 M$ and $\dim M \geqslant 6$, \cite{FarrellFibring1971}, Farrell gave necessary and sufficient $K$-theoretic conditions for an algebraic fibration of $M$ to lift to a smooth fibration $M \longrightarrow S^1$.

In recent work, Italiano, Martelli, and Migliorini gave the first examples of hyperbolic manifolds of dimension $\geqslant 5$ that algebraically fibre \cite[Theorem 1]{italiano2021hyperbolic}.

\begin{thm*}[Italiano, Martelli, Migliorini]
For $5 \leqslant n \leqslant 8$, there are cusped, finite-volume, hyperbolic $n$-manifolds $M^n$ that algebraically fibre. Moreover, the kernel of the algebraic fibration is finitely presented in dimensions $7$ and $8$.
\end{thm*}

In a followup article \cite{italiano2021hyperbolic5}, they construct a smooth fibration of $M^5$ over the circle, providing the first example of a hyperbolic $5$-manifold that fibres over the circle. The manifolds $M^n$ are constructed via gluings of the remarkable right-angled $n$-dimensional hyperbolic polytopes $P^n$, which exist for $3 \leqslant n \leqslant 8$. The polytopes $P^n$ were introduced by Agol, Long, and Reid \cite{AgolLongBianchi01} and have many interesting properties; for more information, we refer the reader to the work of Potyagailo--Vinberg \cite{PotyagailoVinberg05} and Everitt--Ratcliffe--Tschantz \cite{EverittRatcliffe2012}.

Now that examples of hyperbolic $5$-manifolds fibring over the circle are known, a natural direction of research is to look for fibring hyperbolic $7$-manifolds (the Chern--Gauss--Bonnet theorem implies there are no even-dimensional examples), with the most obvious candidates being $M^7$ or one of its finite covers. The main result of this paper does not confirm this, though it does provide strong evidence that a finite cover of $M^7$ fibres over the circle.

\begin{manualtheorem}{\ref{thm:mfldFib}}
    There exists a finite-volume, cusped, hyperbolic $7$-manifold $X^7$ that algebraically fibres with finitely presented kernel of type $\mathtt{FP}(\Q)$. The manifold $X^7$ is a finite cover of the manifold $M^7$ constructed in \cite{italiano2021hyperbolic}.
\end{manualtheorem}

We explain some of the terminology in the statement of the theorem. If there is an epimorphism $\varphi \colon G \longrightarrow \Z$ and $\ker \varphi$ has a finiteness property $\mathcal{P}$ (see \cref{subsec:finprop} for definitions of finiteness properties), we write \textit{$G$ algebraically fibres with kernel of type $\mathcal{P}$}. Similarly, if $M$ is a manifold, then we write \textit{$M$ algebraically fibres with kernel of type $\mathcal{P}$} if its fundamental group algebraically fibres with kernel of type $\mathcal{P}$.

The strategy to prove \cref{thm:mfldFib} is the following. Some straightforward observations show that $\pi_1(M^7)$ is virtually RFRS and that all of its $\ell^2$-Betti numbers vanish (see \cref{lem:mfldRFRS}). By \cite[Theorem 6.6]{Fisher2021improved}, this implies that $\pi_1(M^7)$ has a finite-index subgroup that fibres with kernel of type $\mathtt{FP}(\Q)$. The problem is that this algebraic fibration need not coincide with the one constructed by Italiano--Martelli--Migliorini, so we cannot immediately conclude \cref{thm:mfldFib}. This is remedied by the following general result, which shows that if a RFRS group has many virtual algebraic fibrations with kernels having different finiteness properties, then there is a single virtual algebraic fibration whose kernel has all of these finiteness properties simultaneously.

\begin{manualprop}{\ref{prop:fibring}}
    Let $G$ be a virtually RFRS group and let $\mathbb F_1, \dots, \mathbb F_k$ be skew-fields. Assume that
    \begin{enumerate}
        \item $G$ is of type $\mathtt F_m$ and of type $\mathtt{FP}_{n_i}(\mathbb F_i)$ for some $m, n_1, \dots, n_k \in \N$;
        \item for $i = 0, 1, \dots, k$ there are finite-index subgroups $H_i \leqslant G$ and epimorphisms $\varphi_i \colon H_i \longrightarrow \Z$ such that $\ker \varphi_0$ is of type $\mathtt F_m$ and $\ker \varphi_i$ is of type $\mathtt{FP}_{n_i}(\mathbb F_i)$ for each $i = 1, \dots, k$.  
    \end{enumerate}
    Then there is a finite-index subgroup $H \leqslant G$ and an epimorphism $\varphi \colon H \longrightarrow \Z$ such that $\ker \varphi$ is of type $\mathtt F_m$ and of type $\mathtt{FP}_{n_i}(\mathbb F_i)$ for each $i = 1, \dots, k$.
\end{manualprop}

\begin{rem*}
The techniques used to prove \cref{thm:mfldFib} also provide examples of finite-volume, cusped, hyperbolic $n$-manifolds $X^n$ for $n = 5,8$ such that
\begin{enumerate}
    \item[(3)] $X^5$ algebraically fibres with kernel of type $\FP(\Q)$;
    \item[(4)] $X^8$ algebraically fibres with finitely presented kernel of type $\FP_3(\Q)$;
\end{enumerate}
where $X^n$ is a finite cover of $M^n$. Statement (4) is sharp in the sense that no finite cover of $M^8$ fibres with kernel of type $\FP_4(\Q)$.

Though we provide new methods for obtaining the results (3) and (4), they are already known. In \cite[Theorem 1]{italiano2021hyperbolic5}, Italiano, Martelli, and Migliorini showed that $M^5$ topologically fibres, and in particular all of its finite covers algebraically fibre with kernel of type $\F$. Moreover, Llosa Isenrich, Martelli, and Py have shown that $M^8$ algebraically fibres with kernel of type $\F_3$ but not $\FP_4(\Q)$ \cite[Example 16, Theorem 24]{isenrich2021hyperbolic}.
\end{rem*}

Finally, we note that the algebraic fibration of $M^6$ can be improved, modulo passing to a finite cover. This is an easy corollary of previous work of the author \cite[Theorem 6.6]{Fisher2021improved} and does not rely on \cref{prop:fibring}.

\begin{manualcor}{\ref{cor:6mfld}}
    There is a finite cover $X^6$ of $M^6$ that algebraically fibres with kernel of type $\mathtt{FP}_2(\Q)$. Moreover, no finite cover of $M^6$ algebraically fibres with kernel of type $\mathtt{FP}_3(\Q)$.
\end{manualcor}

\subsection*{Structure} In \cref{sec:prelims}, we provide some background on the tools that will be used in the proof of the main theorem. We define finiteness properties of groups, introduce $\Sigma$-invariants, the Novikov ring, and $\ell^2$-invariants. We then prove \cref{prop:fibring} and \cref{thm:mfldFib} in \cref{sec:algFib}.

\subsection*{Acknowledgements} The author would like to thank Dawid Kielak for the many helpful conversations and suggestions.

This work has received funding from the European Research Council (ERC) under the European Union's Horizon 2020 research and innovation programme (Grant agreement No. 850930).

%%%                   %%%
%%%   Preliminaries   %%%
%%%                   %%%
\section{Preliminaries}\label{sec:prelims}

\subsection{Finiteness properties} \label{subsec:finprop}

We begin with the definition of the homotopic finiteness properties.

\begin{defn}[Type $\F$]
    A group $G$ is of (i) \textit{type $\F_n$}; (ii) \textit{type $\F_\infty$}; (iii) \textit{type $\F$} if $G$ has a classifying space (i) with finite $n$-skeleton; (ii) with finite $n$-skeleton for every $n \in \N$; (iii) that is compact.
\end{defn}

Note that being of type $\F_1$ (resp.~$\F_2$) is equivalent to being finitely generated (resp.~finitely presented). The following finiteness properties are homological analogues of those defined above.

\begin{defn}[Type $\FP$]
Let $R$ be a ring and $M$ an $R$-module. Then $M$ is of (i) \textit{type $\FP_n(R)$}; (ii) \textit{type $\FP_\infty(R)$}; (iii) \textit{type $\FP(R)$} if there is a projective resolution $P_\bullet \longrightarrow M$ where (i) $P_i$ is finitely generated for $i \leqslant n$; (ii) $P_n$ is finitely generated for every $n \in \N$; (iii) $P_n$ is finitely generated for every $n \in \N$ and there is some $k$ such that $P_n = 0$ for $n \geqslant k$.

A group $G$ is of \textit{type $\mathtt{FP}_n(R)$} if the trivial $RG$-module $R$ is of type $\mathtt{FP}_n(RG)$. Similarly, $G$ is of \textit{type $\mathtt{FP}_\infty(R)$} (resp.~\textit{type $\mathtt{FP}(R)$}) if $R$ is of type $\mathtt{FP}_\infty(RG)$ (resp.~$\mathtt{FP}(RG)$).
\end{defn}

It is not hard to prove that if a group is of type $\F_n$ then it is of type $\FP_n(R)$ for any ring $R$, and similarly for the other finiteness properties defined above. Moreover, if $G$ is of type $\FP_1(R)$ for some nonzero ring $R$, then $G$ is finitely generated. However, it is not true that being of type $\FP_n(R)$ implies being of type $F_n$ for $n \geqslant 2$, as was first shown by Bestvina and Brady in \cite{BestvinaBrady}.

\subsection{\texorpdfstring{$\Sigma$}{}-invariants and the Novikov ring}

For the remainder of \cref{sec:prelims}, $R$ will denote an associative, unital ring, and $G$ will denote a finitely generated group.

\begin{defn}
A \textit{character} is a homomorphism $\chi \colon G \longrightarrow \R$, where $\R$ denotes the additive group of real numbers. We define an equivalence relation on the set of nonzero characters by declaring characters $\chi$ and $\chi'$ to be equivalent if and only if there is a positive constant $\alpha > 0$ such that $\alpha \chi = \chi'$. The equivalence class of $\chi$ is denoted $[\chi]$. The \textit{character sphere} $S(G)$ of $G$ is the set of equivalence classes of characters on $G$. Note that $S(G)$ naturally has the topology of the sphere $S^{n-1}$, where $n$ is the first Betti number of $G$. When $N \leqslant G$ is a subgroup, we will also be interested in the subspace $S(G,N) := \{ [\chi] \in S(G) : \chi(N) = 0 \}$.
\end{defn}

\begin{defn}[$\Sigma$-invariants]
Let $M$ be an $RG$-module. Define
\[
\Sigma^n_R(G;M) = \{ [\chi] \in S(G) : M \ \text{is of type} \ \FP_n(RG_\chi)  \},
\]
where $G_\chi$ is the monoid $\{ g \in G : \chi(g) \geqslant 0 \}$.
\end{defn}

The invariants $\Sigma^n_R(G;M)$ are generalisations of the classical Bieri--Neumann--Strebel invariant \cite{BNSinv87} and its higher-dimensional analogues \cite{BieriRenzValutations}. The only difference is that we work over a general ring $R$, while the higher BNS invariants are defined over $\Z$. We note the following important facts about the invariants $\Sigma^n_R(G;M)$.

\begin{lem}[Bieri--Renz]\label{lem:hmlgBNS}
\begin{enumerate}
    \item $\Sigma^n_R(G;M)$ is open in $S(G)$ for all $n \in \N$ \cite[Theorem A]{BieriRenzValutations};
    \item if $G$ is of type $\FP_n(R)$ and $N \triangleleft G$ is a normal subgroup with Abelian quotient $G/N$, then $N$ is of type $\FP_n(R)$ if and only if $S(G,N) \subseteq \Sigma^n_R(G;R)$ \cite[Theorem B]{BieriRenzValutations}.
\end{enumerate}
\end{lem}

In \cite{BieriRenzValutations}, Bieri and Renz only state the above result in the case $R = \Z$, however their proof carries over without modification to the more general setting.

\begin{defn}[Novikov ring]
Let $\chi \colon G \to \R$ be a character. Then the associated \textit{Novikov ring} is
\[
    \widehat{RG}^\chi = \left\{ \sum_{g \in G} r_g g : \chi\inv(\left]-\infty, t\right]) \cap \{g \in G : r_g \neq 0\} \ \text{is finite for all} \ t \in \R \right\}.
\]
The elements of $\widehat{RG}^\chi$ should be thought of as formal sums of elements of $G$ with coefficients in $R$ that are finitely supported below every height, where $\chi$ is the height function. We give $\widehat{RG}^\chi$ a ring structure by defining $rg + r'g := (r + r')g$ and $rg \cdot r'g' := rr' gg'$ for $r,r' \in R$, $g,g' \in G$, and extending multiplication to all of $\widehat{RG}^\chi$ in the obvious way.
\end{defn}

The $\Sigma$-invariants can be characterised in terms of vanishing group homology with coefficients in the Novikov ring.

\begin{lem}[Bieri--Renz, Schweitzer]\label{lem:sigmaChar}
Let $M$ be a left $RG$-module of type $\FP_n(RG)$ and let $\chi \colon G \to \R$ be a non-trivial character. The following are equivalent:
\begin{enumerate}
    \item\label{item:SigmaInv} $[\chi] \in \Sigma_R^n(G;M)$;
    \item\label{item:TorCond} $\Tor_p^{RG}(\widehat{RG}^\chi, M) = 0$ for all $p \leqslant n$.
\end{enumerate}
\end{lem}

A proof of \cref{lem:sigmaChar} can be found in \cite[Theorem 5.3]{Fisher2021improved}. The result with $R = \Z$ follows from work of Schweitzer \cite[Appendix]{BieriDeficiency} and Bieri--Renz \cite{BieriRenzValutations}, and their arguments carry over without modification to the case of a general ring $R$.

In his thesis, Renz defined homotopical analogues $\prescript{*}{}{\Sigma}^m(G)$ of the $\Sigma$-invariants which we now describe. Let $G$ be of type $\F_m$ and suppose that $X$ is a $K(G,1)$ with finite $m$-skeleton. Let $\chi \colon G \longrightarrow \R$ be a character and let $Y \longrightarrow X$ be the covering space corresponding to $\ker \chi$. Then, it is possible to find a continuous height function $f \colon Y \longrightarrow \R$ such that $f(g \cdot y) = \chi(g) + f(y)$ for all $g \in G$ and $y \in Y$. Let $Y_{\chi,r} := \{ y \in Y : f(y) \geqslant -r \}$.

\begin{defn}[Homotopical $\Sigma$-invariants]
The $m$th homotopical $\Sigma$-invariant of $G$ is the subset $\prescript{*}{}{\Sigma}^m(G)$ of $S(G)$ defined by the following property: $[\chi] \in \prescript{*}{}{\Sigma}^m(G)$ if and only if  there is some $r \geqslant 0$ such that the inclusion-induced homomorphism $\pi_i(Y_{\chi,0}) \longrightarrow \pi_i(Y_{\chi, r})$ is trivial for $i < m$.
\end{defn}

Importantly for our purposes, \cref{lem:hmlgBNS} carries over to the homotopical context.

\begin{lem}[Renz] \label{lem:htpbns} Let $G$ be of type $\F_m$. Then
\begin{enumerate}
    \item $\prescript{*}{}{\Sigma}^m(G)$ is open in $S(G)$ for all $m \in \N$ \cite[Satz A]{RenzThesis};
    \item if $N \triangleleft G$ is a normal subgroup with Abelian quotient $G/N$, then $N$ is of type $\F_m$ if and only if $S(G,N) \subseteq \prescript{*}{}{\Sigma}^m (G)$ \cite[Satz C]{RenzThesis}.
\end{enumerate}
\end{lem}

\subsection{\texorpdfstring{$\ltwo$}{}-invariants and the Linnell skew-field}\label{subsec:l2}

We briefly summarise the construction of the Linnell ring of a group and explain how it is used to compute $\ltwo$-homology.

\begin{defn}[von Neumann algebra]\label{def:vonNalg}
Let $G$ be a countable group and let $\ell^2(G)$ denote the Hilbert space of square-summable complex series over $G$:
\[
    \ell^2(G) := \left\{  \sum_{g \in G} \lambda_g g \ : \ \lambda_g \in \C, \sum_{g \in G} \lvert\lambda_g\rvert^2 < \infty \right\}.
\]
The \textit{von Neumann algebra} $\mathcal{N}(G)$ of $G$ is the algebra of bounded $G$-equivariant operators on $\ltwo(G)$. Here $G$ acts on $\ltwo(G)$ by left multiplication and the operators act on the right.
\end{defn}

The elements of the group ring $\Q G$ act by right multiplication on $\ltwo(G)$ and therefore can be viewed as $G$-equivariant operators. One easily checks that these operators are bounded; thus, we have an embedding $\Q G \longhookrightarrow \mathcal{N}(G)$. An important fact about $\mathcal{N}(G)$ is that it satisfies the \textit{Ore condition} \cite[Theorem 8.22(1)]{Luck02}. We will not give a definition of the Ore condition; it will suffice to know that if a ring $R$ satisfies the Ore condition, then it can be embedded in a ring $\Ore(R)$ in such a way that all the non-zero divisors of $R$ are mapped to invertible elements of $\Ore(R)$. We refer the interested reader to \cite[Section 4.4]{PassmanGrpRng} for a detailed account of the Ore condition.

\begin{defn}[Linnell ring, Atiyah conjecture]\label{def:linnellring}
The \textit{Linnell ring} of a group $G$, denoted $\mathcal{D}(G)$ is the division closure of (the image of) $\Q G$ in $\Ore(\mathcal{N}(G))$. A torsion-free group $G$ satisfies the \textit{Atiyah conjecture} if $\mathcal{D}(G)$ is a skew-field, that is, if every nonzero element of $\mathcal{D}(G)$ is invertible.
\end{defn}

The formulation of the Atiyah conjecture given here is equivalent to the strong Atiyah conjecture over $\Q$. Linnell showed that the formulation in \cref{def:linnellring} implies the strong Atiyah conjecture over $\Q$, and the details of the reverse implication can be found in L\"uck's book \cite[Section 10]{Luck02}. We note here that it follows from work of Schick \cite{SchickL2Int2002} that RFRS groups (see \cref{def:RFRS}) satisfy the Atiyah conjecture.

\begin{defn}[$\ltwo$-homology and $\ltwo$-Betti numbers]
Let $G$ be a torsion-free group satisfying the Atiyah conjecture. We define the \textit{$\ltwo$-homology} of $G$ to be $H_\bullet(G;\mathcal{D}(G))$, where $\mathcal{D}(G)$ is viewed as a left $\Q G$-module. The $n$th \textit{$\ltwo$-Betti number} of $G$ is 
\[
    \btwo{n}(G) := \dim_{\mathcal{D}(G)} H_n(G;\mathcal{D}(G)),
\]
where we view $H_n(G;\mathcal{D}(G))$ as a left $\mathcal{D}(G)$-module.
\end{defn}

\begin{rem}
The $\ltwo$-Betti numbers of a group are well-defined, since modules over skew-fields have well-defined dimensions. We also remark that this is not the usual definition of $\ltwo$-invariants, though it is the one that will be used below. The definitions coincide by \cite[Lemma 10.28]{Luck02}.
\end{rem}

Let $G\ab = G/[G,G]$ be the Abelianisation of $G$.
\begin{defn} \label{def:RFRS}
A group $G$ is \textit{residually finite rationally solvable (RFRS)} if 
\begin{enumerate}
\itemsep-0.2em
\item there is a chain $G = G_0 \geqslant G_1 \geqslant G_2 \geqslant \cdots$ of finite-index normal subgroups of $G$ such that $\bigcap_{i=0}^\infty G_i = 1$;
\item $\ker(G_i \longrightarrow \Q \otimes G_i^{\mathrm{ab}}) \leqslant G_{i+1}$ for every $i \geqslant 0$.
\end{enumerate}
\end{defn}

The following theorem is the main ingredient in the proof of \cref{prop:fibring}. 

\begin{thm}[Kielak, Jaikin-Zapirain] \label{thm:kielak}
Let $G$ be a finitely generated RFRS group, let $\mathbb F$ be a skew-field, and let $n \in \N$. Let $C_\bullet$ denote a chain complex of free $\mathbb F G$-modules such that for every $p \leqslant n$ the module $C_p$ is finitely generated and $H_p(\mathcal D_{\mathbb F G} \otimes_{\mathbb F G} C_\bullet) = 0$. Then, there exist a finite-index subgroup $H \leqslant G$ and an open subset $U \subseteq S(H)$ such that
\begin{enumerate}
    \item the closure of $U$ contains $S(G)$;
    \item $U$ is invariant under the antipodal map;
    \item $H_p(\widehat{\mathbb F H}^\psi \otimes_{\mathbb F H} C_\bullet) = 0$ for every $p \leqslant n$ and every $\psi \in U$.
\end{enumerate}
\end{thm}

Here, $\mathcal D_{\mathbb F G}$ is a skew-field containing $\mathbb F G$ known as the \textit{Hughes-free division ring} of $\mathbb FG$. Importantly for us, $\mathcal D_{\mathbb FG}$ exists for any skew-field $\mathbb F$ and $\mathcal D_{\Q G}$ coincides with $\mathcal D(G)$ whenever $G$ is a RFRS group \cite[Corollary 1.3]{JaikinUniversal2021}. The $\mathbb F = \Q$ case of \cref{thm:kielak} is due to Kielak \cite[Theorem 5.2]{KielakRFRS} and the generalisation to arbitrary skew-fields follows from Jaikin-Zapirain's appendix to \cite{JaikinUniversal2021}.

\begin{rem}
    Jaikin-Zapirain does not state \cref{thm:kielak} explicitly, so we outline how his results imply the generalisation of Kielak's theorem. We hope this extended remark will also make it easier for the interested reader to find the analogues of Kielak's results in Jaikin-Zapirain's appendix. Let $G$ be a RFRS group with a residual chain $(G_i)$ witnessing the RFRS property. Kielak begins by defining \textit{rich} subsets of $H^1(G_i; \R)$ and proves some of their basic properties in \cite[Lemmas 4.5, 4.5, 4.6]{KielakRFRS}. Jaikin-Zapirain uses a slightly more restrictive definition of rich subsets of $H^1(G_i; \R)$, though there is no essential difference and Kielak's proofs still hold in this setting (see \cite[Section 5.1]{JaikinUniversal2021}). 

    Kielak introduces the following central concept in \cite[Definition 4.7]{KielakRFRS}: An element $x \in \mathcal D(G)$ is \textit{well representable} if there is some $n \in \N$ and there are rich subsets $U_i \subseteq H^1(G_i;\R)$ for all $i \geqslant n$ such that $x$ can be interpreted as an element of the \textit{twisted} group ring $\widehat{\Q G_i}^\varphi[G/G_i]$ for every character $\varphi \in U_i$. The way to make this interpretation precise is by realising both $\widehat{\Q G_i}^\varphi$ and $\mathcal D(G_i)$ as subrings of a common larger ring and noting that $\mathcal D(G_i)[G/G_i] \cong \mathcal D(G)$, where $\mathcal D(G_i)[G/G_i]$ is a twisted group ring. Then $x$ is well representable if it lies in  $(\mathcal D(G_i) \cap \widehat{\Q H_i}^\varphi)[G/G_i]$.

    For an arbitrary skew-field $\mathbb F$, Jaikin-Zapirain analogously describes how to view $\widehat{\mathbb F G_i}^\varphi$ and $\mathcal D_{\mathbb F G_i}$ as subrings of a common ring in \cite[Section 5.2]{JaikinUniversal2021} and defines well representability of elements of $\mathcal D_{\mathbb F G}$ at the beginning of \cite[Section 5.5]{JaikinUniversal2021}. The main result (\cite[Theorem 4.13]{KielakRFRS} for $\mathbb F = \Q$ and \cite[Thorem 5.10]{JaikinUniversal2021} for arbitrary $\mathbb F$) is that every element of $\mathcal D_{\mathbb F G}$ is well representable. Equipped with this structure theorem for $\mathcal D_{\mathbb F G}$, Kielak's proof of \cref{thm:kielak} still holds after making the replacements $\mathbb Q \rightsquigarrow \mathbb F$ and $\mathcal D(G) \rightsquigarrow \mathcal D_{\mathbb F G}$.
\end{rem}

The last thing we will need to know about Hughes-free division rings is their relation to algebraic fibring of RFRS groups, given by the following result of the author. 

\begin{thm}[{\cite[Theorem 6.6]{Fisher2021improved}}]\label{thm:b2rfrs}
    Let $\mathbb F$ be a skew-field and let $G$ be a virtually RFRS group of type $\FP_n(\mathbb F)$. Then there is a finite-index subgroup $H \leqslant G$ admitting an epimorphism to $\Z$ with kernel of type $\FP_n(\mathbb F)$ if and only if $H_p(G; \mathcal{D}_{\mathbb F G}) = 0$ for all $p = 0, \dots, n$.
\end{thm}

\begin{rem}\label{rem:notation}
    There is a slight abuse of notation in the statement of \cref{thm:b2rfrs}. Namely, $\mathcal D_{\mathbb F G}$ may not exist when $G$ is only \textit{virtually} RFRS. For example, if $G$ has torsion then $\mathbb F G$ does not embed into any skew-field. However, if $H$ is RFRS, then $\mathcal D_{\mathbb F H}$ does exist and the Betti numbers $b_p(H; \mathcal D_{\mathbb FH}) = \dim_{\mathcal D_{\mathbb FH}} H_p(H; \mathcal{D}_{\mathbb F H})$ satisfy 
    \[
        b_p(H'; \mathcal D_{\mathbb FH'}) = [H:H'] \cdot b_p(H; \mathcal D_{\mathbb FH})
    \]
    whenever $H'$ is a finite-index subgroup of $H$ \cite[Lemma 6.3]{Fisher2021improved}. Thus, when $G$ is virtually RFRS, we interpret $H_p(G; \mathcal{D}_{\mathbb F G}) = 0$ to mean that $G$ has some finite-index RFRS subgroup $H$ such that $H_p(H; \mathcal{D}_{\mathbb F H}) = 0$.
\end{rem}

%%%                  %%%
%%%   Main results   %%%
%%%                  %%%
\section{Main results}\label{sec:algFib}

The manifolds $M^n$ constructed by Italiano--Martelli--Migliorini in \cite{italiano2021hyperbolic} are special cases of the following general construction, introduced by Vesnin in \cite{Vesnin87}. Let $P$ be an $n$-dimensional finite-volume right-angled polyhedron in $\mathbb{H}^n$ with real and ideal vertices and let $\Gamma$ be the right-angled Coxeter group generated by the reflections across the $(n-1)$-cells of $P$. This is the group with presentation
\[
    \pres{ F_1, \dots, F_k }{ [F_i, F_j] \ \text{if and only if $F_i$ and $F_j$ meet along an $(n-2)$-cell}}
\]
where the generators $F_i$ range over the $(n-1)$-cells of $P$. Colour the $(n-1)$-cells using the set of colours $\{1,2, \dots, m\}$ so that adjacent cells have different colours and fix a generating set $\{e_1, \dots, e_m\}$ of $(\Z/2)^m$. There is then a homomorphism $\varphi \colon \Gamma \longrightarrow (\Z/2)^m$ determined by $F_i \longmapsto e_c$ if and only if the colour of $F_i$ is $c$. Then $\Gamma' := \ker \varphi$ acts freely on $\mathbb{H}^n$ and $M = \mathbb{H}^n / \Gamma'$ is a finite-volume hyperbolic $n$-manifold.

Note that $\pi_1(M) = \Gamma'$ is a subgroup of a right-angled Coxeter group and is therefore virtually RFRS by \cite[Theorem 2.2]{AgolCritVirtFib}. Moreover, $\pi_1(M)$ is a lattice in $\Isom(\mathbb{H}^n)$ and therefore its $\ell^2$-Betti numbers vanish except in the middle dimension \cite[Corollary 5.16(1)]{Luck02}. Since the manifolds $M^n$ of \cite{italiano2021hyperbolic} are special cases of the above construction, we have the following facts.

\begin{lem}\label{lem:mfldRFRS}
    For $5 \leqslant n \leqslant 8$,
    \begin{enumerate}
        \item $\pi_1(M^n)$ is virtually RFRS;
        \item $\btwo{p}(\pi_1(M^n)) = 0$ if $p \neq n/2$ and $\btwo{n/2}(\pi_1(M^n)) \neq 0$ for $n = 6,8$.
    \end{enumerate}
\end{lem}

We note the following immediate application of \cref{thm:b2rfrs} and \cref{lem:mfldRFRS} to algebraic fibring of $M^6$, the $6$-dimensional, finite-volume, cusped, hyperbolic manifold constructed in \cite{italiano2021hyperbolic}.

\begin{cor}\label{cor:6mfld}
    There is a finite cover $X^6$ of $M^6$ that algebraically fibres with kernel of type $\mathtt{FP}_2(\Q)$. Moreover, no finite cover of $M^6$ algebraically fibres with kernel of type $\mathtt{FP}_3(\Q)$.
\end{cor}

\begin{prop}\label{prop:fibring}
    Let $G$ be a virtually RFRS group and let $\mathbb F_1, \dots, \mathbb F_k$ be skew-fields. Assume that
    \begin{enumerate}
        \item $G$ is of type $\mathtt F_m$ and of type $\mathtt{FP}_{n_i}(\mathbb F_i)$ for some $m, n_1, \dots, n_k \in \N$;
        \item for $i = 0, 1, \dots, k$ there are finite-index subgroups $H_i \leqslant G$ and epimorphisms $\varphi_i \colon H_i \longrightarrow \Z$ such that $\ker \varphi_0$ is of type $\mathtt F_m$ and $\ker \varphi_i$ is of type $\mathtt{FP}_{n_i}(\mathbb F_i)$ for each $i = 1, \dots, k$.  
    \end{enumerate}
    Then there is a finite-index subgroup $H \leqslant G$ and an epimorphism $\varphi \colon H \longrightarrow \Z$ such that $\ker \varphi$ is of type $\mathtt F_m$ and of type $\mathtt{FP}_{n_i}(\mathbb F_i)$ for each $i = 1, \dots, k$.
\end{prop}

\begin{proof}
    We proceed by induction on $k$, with the base case $k = 0$ being trivial. Suppose that $k > 0$. By induction, there is a finite-index subgroup $H' \leqslant G$ and an epimorphism $\varphi' \colon H' \longrightarrow \Z$ with $\ker \varphi$ of type $\mathtt F_m$ and of type $\mathtt{FP}_{n_i}(\mathbb F_i)$ for each $i = 1, \dots, k-1$. Moreover, there is a finite-index subgroup $H_k \leqslant G$ and an epimorphism $\varphi_k \colon H_k \longrightarrow \Z$ with kernel of type $\mathtt{FP}_{n_k}(\mathbb F_k)$. By passing to finite-index subgroups if necessary, we will assume that $G = H' = H_k$ and that $G$ is RFRS.

    Since $G$ is of type $\mathtt{FP}_{n_k}(\mathbb F_k)$, there is a free resolution $C_\bullet \longrightarrow \mathbb F_k$ by $\mathbb F_k G$-modules such that $C_p$ is finitely generated for all $p \leqslant n_k$. Moreover, by \cref{thm:b2rfrs},
    \[
        H_p(G; \mathcal D_{\mathbb F_k G}) = H_p(\mathcal D_{\mathbb F_k G} \otimes_{\mathbb F_k G} C_\bullet) = 0
    \]
    for $p \leqslant n_k$. By \cref{thm:kielak}, there is a finite-index subgroup $H \leqslant G$ and an open set $U \subseteq S(H)$ satisfying $S(G) \subseteq \overline{U}$, $U = -U$, and $H_p(\widehat{\mathbb F_k H}^\psi \otimes_{\mathbb F_k H} C_\bullet) = 0$ for all $p \leqslant n_k$ and all $\psi \in U$.

    Let $\varphi := \varphi'|_H$. Then $\varphi' \in S(G) \subseteq \overline{U}$ and $\ker \varphi$ has all the finiteness properties of $\ker \varphi'$ because $H$ is a finite-index subgroup of $G$. By Lemmas \ref{lem:hmlgBNS} and \ref{lem:htpbns},
    \[
        S(H, \ker \varphi) = \{ [\pm \varphi] \} \subseteq \prescript{*}{}{\Sigma}^m (H) \cap \bigcap_{i=1}^{k-1} \Sigma_{\mathbb F_i}^{n_i} (H, \mathbb F_i) =: \Sigma,
    \]
    and since $\Sigma$ is open in $S(H)$, there is an antipodally symmetric open neighbourhood $V \subseteq \Sigma$ containing $\{[\pm \varphi]\}$. Since $S(G) \subseteq \overline{U}$ and $U$, $V$ are antipodally symmetric open sets, there is an epimorphism $\psi \colon H \longrightarrow \Z$ such that $\{[\pm \psi]\} \subseteq U \cap V$.

    We claim that $\ker \psi$ has the desired finiteness properties. First,
    \[
        \{[\pm \psi]\} = S(H, \ker \psi) \subseteq V \subseteq \prescript{*}{}{\Sigma}^m (H) \cap \bigcap_{i=1}^{k-1} \Sigma_{\mathbb F_i}^{n_i} (H, \mathbb F_i),
    \]
    so $\ker \psi$ is of type $\mathtt F_m$ by \cref{lem:htpbns} and of type $\mathtt{FP}_{n_i}(\mathbb F_i)$ for all $i = 1, \dots, k-1$ by \cref{lem:hmlgBNS}. Second, since $\{[\pm \psi]\} \subseteq U$, we have $\Tor_p^{\mathbb F_k H}(\widehat{\mathbb F_k H}^{\pm \psi}, \mathbb F_k) = 0$ for $p \leqslant n_k$ and therefore $[\pm \psi] \in \Sigma_{\mathbb F_k}^{n_k}(H; \mathbb F_k)$ by \cref{lem:sigmaChar}. But then $\ker \psi$ is of type $\mathtt{FP}_{n_k}(\mathbb F_k)$ by \cref{lem:hmlgBNS}.
\end{proof}

We are now ready to prove the main result.

\begin{thm}\label{thm:mfldFib}
    There exists a finite-volume, cusped, hyperbolic $7$-manifold $X^7$ that algebraically fibres with finitely presented kernel of type $\mathtt{FP}(\Q)$. The manifold $X^7$ is a finite cover of the manifold $M^7$ constructed in \cite{italiano2021hyperbolic}.
\end{thm}

\begin{proof}
    Let $G = \pi_1(M^7)$. Note that $G$ is of type $\mathtt F$ and, by \cref{lem:mfldRFRS}, $G$ is virtually RFRS and $\btwo{p}(G) = 0$ for all $p$. By \cref{thm:b2rfrs}, $G$ virtually algebraically fibres with kernel of type $\mathtt{FP}_m(\Q)$ for $m$ arbitrarily large. Since $G$ has finite cohomological dimension, so do all of its subgroups. Therefore, taking $m \geqslant \operatorname{cd}(G)$ guarantees that $G$ is virtually algebraically fibred with kernel of type $\FP(\Q)$ \cite[VIII Proposition 6.1]{BrownGroupCohomology}. From \cref{prop:fibring} applied to the case $k = 1$ and $\mathbb F_1 = \Q$ and the fact that $G$ algebraically fibres with finitely presented kernel \cite[Theorem 1]{italiano2021hyperbolic}, we conclude that there is a finite-index subgroup $H \leqslant G$ and an epimorphism $\varphi \colon H \longrightarrow \Z$ such that $\ker \varphi$ is finitely presented and of type $\mathtt{FP}_n(\Q)$ for $n$ arbitrarily large. But again, $H$ has finite cohomological dimension and therefore so does $\ker \varphi$. Thus $\ker \varphi$ is of type $\mathtt{FP}(\Q)$, provided we take $n \geqslant \operatorname{cd}(\ker \varphi)$. Take $X^7$ to be the covering space of $M^7$ corresponding to $H$. \qedhere
\end{proof}

\begin{rem}
    Recall that if a group is finitely presented and of type $\mathtt{FP}_n(\Z)$, then it is of type $\mathtt F_n$. Promoting the algebraic fibration of $\pi_1(X^7)$ to one with kernel of type $\mathtt F$ would get us very close to proving that $M^7$ virtually fibres over the circle, thanks to the work of Farrell \cite{FarrellFibring1971}. However, it is not even true in general that finitely presented groups of type $\mathtt{FP}(\Q)$ are of type $\mathtt{FP}_\infty(\Z)$. For instance, let $L$ be a simply-connected $\Q$-acyclic flag complex with nonzero integral homology. Then the Bestvina--Brady group $BB_L$ is finitely presented and of type $\mathtt{FP}(\Q)$, but $BB_L$ is not of type $\mathtt{FP}_\infty(\Z)$ by \cite[Main Theorem]{BestvinaBrady}. Thus, \cref{thm:mfldFib} does not imply the existence of kernels with higher homotopical finiteness properties than $\mathtt F_2$.
\end{rem}

\bibliographystyle{alpha}
\bibliography{bib}

\end{document}